\documentclass[11pt]{amsart}
\usepackage[T1]{fontenc}
\numberwithin{equation}{section}

\usepackage{float}
\usepackage{tikz-cd}
\usepackage{amsmath}%
\usepackage{amsfonts}%
\usepackage{amsthm}
\usepackage{amssymb}%
\usepackage{graphicx}
\usepackage{hyperref}

\usepackage{float}
\usepackage{tikz-cd}
\usepackage{amsmath}%
\usepackage{amsfonts}%
\usepackage{amsthm}
\usepackage{amssymb}%
\usepackage{graphicx}
\usepackage{hyperref}
\usepackage{enumerate}
\usepackage{comment}

\usepackage{geometry}
\usepackage{layout}
\usepackage[utf8]{inputenc}
\usepackage{amsmath}
\usepackage{graphicx}
\usepackage{amsmath}
\usepackage{amssymb}
\usepackage{amsfonts}
\usepackage{graphicx}
\usepackage{url}
\usepackage{MnSymbol}
\usepackage{arydshln}
\usepackage{outlines}
\usepackage{tikz-cd}

\usepackage{subfig}
\usepackage{amsthm}
\theoremstyle{plain}

\newcommand{\Fbar}{\overline{\mathbb{F}}}

\newcommand{\Z}{\mathbb{Z}}

\renewcommand{\epsilon}{\varepsilon}

\newcommand{\im}{\mathrm{im}}

\usepackage{multirow}

\usepackage{xcolor}

\newtheorem{thm}{Theorem}[section]

\newtheorem*{thm*}{Theorem}
\newtheorem*{cor*}{Corollary}
\newtheorem*{prop*}{Proposition}

\newtheorem{lem}[thm]{Lemma}
\newtheorem{prop}[thm]{Proposition}
\newtheorem{cor}[thm]{Corollary}

\theoremstyle{definition}

\theoremstyle{remark}
\newtheorem{rem}[thm]{Remark}
\newtheorem{exmp}[thm]{Example}

\begin{document}

\author{Du\v san Dragutinovi\'c}
\address{}
\email{dusandragutinovic1@gmail.com}

\title[$n$-nacci numbers, partitions, and binomial coefficients]{An identity relating $n$-nacci numbers, partitions, and products of binomial coefficients}

\maketitle

\begin{abstract}
We study the combinatorial properties of final types, which are certain non-decreasing sequences of integers, together with the partitions naturally associated with them. As a consequence, we obtain an identity expressing the $n$-nacci numbers as sums of products of binomial coefficients over these partitions, generalizing the classical identity for $n = 2$ that expresses Fibonacci numbers in this way. We also examine how the partial order on the set of all partitions of a fixed integer induced by the ordering of final types compares with two natural partial orders on the same set.
\end{abstract}

\section{Introduction}
\label{sec:intro}
In this note, we study the relations between final types $\nu$, which are certain non-decreasing sequences of length $g\geq 1$ with maximal step $1$, and the partitions $\delta = \delta_{\nu}$ of~$g$ that are defined from them; see Section \ref{sec:sec2} for more details. Our aim is to highlight several combinatorial features of this correspondence and to draw attention to some consequences that seem to be of independent interest.

As a first consequence of this relation, we obtain identities involving products of binomial coefficients and the $n$-nacci numbers, generalizing a well-known result for $n = 2$ concerning the representation of Fibonacci numbers as sums of specific binomial coefficients. 
More precisely, given $g \geq 1$ and $n \geq 1$, let $\Tilde{F}_n(g)$ be  the sequence of $n$-nacci numbers, that is, $$\Tilde{F}_n(g) = \Tilde{F}_n(g - 1) + \ldots + \Tilde{F}_n(g - n),\quad  g \geq n + 1$$  with initial conditions $$\Tilde{F}_n(1) = \ldots = \Tilde{F}_n(n - 1) = 0\quad \text{and}\quad \Tilde{F}_n(n) = 1.$$
In the special case $n = 2$, recall that $\Tilde{F}_2(g)$ is equal to the $g$-th Fibonacci number. Namely, $\Tilde{F}_2(g) = \mathrm{Fibonacci}(g),$ where $$\mathrm{Fibonacci}(g) = \mathrm{Fibonacci}(g - 1) + \mathrm{Fibonacci}(g-2) \text{ for }g \geq 2$$ with 
$$\mathrm{Fibonacci}(1) = 0\quad \text{and}\quad \mathrm{Fibonacci}(2) = 1.$$
It is a classical identity, already observed by Lucas in \cite{luc}, that $\Tilde{F}_2(g) = \mathrm{Fibonacci}(g)$ can be expressed as a sum of binomial coefficients for any $g\geq 2$. For example,
$$5 = \Tilde{F}_2(6) = {4 \choose 0} + {3 \choose 1} + {2 \choose 2} \text{ and } 144 = \Tilde{F}_2(13) = {11 \choose 0} + {10 \choose 1} + {9 \choose 2} + {8 \choose 3} + {7 \choose 4} + {6 \choose 5}.$$
More generally, as noted in~\cite[p.~7]{luc}, one has $\mathrm{Fibonacci}(g+2) = \sum_{i = 0}^{\left \lfloor {g}/{2} \right \rfloor} {g - i \choose i}$, i.e.,
\begin{equation}
\Tilde{F}_2(g + 2) = \sum_{i = 0}^{\left \lfloor \frac{g}{2}\right \rfloor} {g - i \choose i} \text{ for any } g\geq 0.
\label{eqn:fibo_iden}
\end{equation}

Below we observe the following identity, which generalizes \eqref{eqn:fibo_iden} from $n = 2$ to all $n \geq 2$. Although we expect that the case $n \geq 3$ is known in at least some instances, the lack of a suitable reference leads us to include this elegant formula here. For some similar identities, including the generating function of generalized Fibonacci and $n$-nacci numbers and partitions, we refer to \cite{yang}, \cite{huangliu}, \cite{sills}, and \cite{komatsu}.

\begin{thm*} For any $g \geq 1$ and $n \geq 1$, the following identity holds: 
$$\Tilde{F}_n(g + n) = \sum_{\substack{0 \leq \delta_n \leq \ldots \leq \delta_2 \leq \delta_1 \leq g \\ \delta_1 + \delta_2 + \ldots + \delta_n = g}} {\delta_1 \choose \delta_2}{\delta_2 \choose \delta_3}\ldots {\delta_{n -1} \choose \delta_{n}}.$$
\end{thm*}

A partition $\delta = \delta_1\ldots \delta_n$ of $g$ consists of $1 \leq \delta_n \leq \delta_{n - 1} \leq \ldots \leq \delta_1 \leq g$ such that $\delta_1 + \ldots + \delta_n = g$. In this case, we write $\delta \vdash g$ and denote the length of $\delta$, that is, the number of elements appearing in this partition, by $|\delta| = n$. For example, if $\delta = 3111$ and $\delta' = 832$, then $\delta \vdash 6,$ $|\delta| = 4, \delta' \vdash 13$, and $|\delta'| = 3$.

A consequence of the preceding theorem for~${n = g}$ provides a neat connection between several combinatorial notions, as shown in the corollary below.

\begin{cor*} For any $g \geq 1$, the following identity holds: 
$$\sum_{\delta \vdash g} \prod_{i = 1}^{|\delta| - 1} {\delta_i \choose \delta_{i + 1}} = 2^{g - 1}.$$
\end{cor*}
\noindent
In Corollary \ref{cor:nice_identity_g_g}, we offer an alternative, purely combinatorial proof of the previous corollary, which does not rely on the preceding theorem.

In addition to the material mentioned above, we continue our discussion of the relationship between final types and the partitions defined from them by studying three partial orderings on the set of all partitions of a fixed integer: the order induced by the lexicographic order on final types, $\leq_{ft}$, the order defined in terms of partial partitions, $\leq_{pp}$, and the dominance order, $\leq_{do}$; see Section \ref{sec:order} for more details. In particular, we obtain the following result in Proposition \ref{prop:last}, which shows an interesting property of $\leq_{ft}$, placing it between the two classically studied orders.

\begin{prop*} For any two partitions $\delta \vdash g$ and $\delta' \vdash g$, it holds that $$\delta \leq_{pp} \delta' \implies \delta \leq_{ft} \delta' \implies \delta \leq_{do} \delta'.$$
In general, $\delta \leq_{ft} \delta' \not \implies \delta \leq_{pp} \delta'$ and $\delta \leq_{do} \delta' \not \implies \delta \leq_{ft} \delta'$.
\end{prop*}

\emph{Outline.} In Section \ref{sec:sec2}, we present initial results relating final types and partitions. Furthermore, in Subsection \ref{subsec:motivation}, we briefly explain our motivation for studying these combinatorial invariants, which comes from algebraic geometry and the study of $p$-torsion of abelian varieties in characteristic $p$. In Section \ref{sec:sec3}, we prove the theorem stated above. Finally, in Section \ref{sec:order}, we consider another viewpoint motivated by the discussion in Subsection \ref{subsec:motivation}, concerning the three partial orders on the set of all partitions, and we compare them.

\subsection*{Acknowledgment} This project originated at, and was partially supported by, the Mathematical Institute of Utrecht University. Some of the initial results appeared in earlier versions of \cite{dra}, and the author would like to thank Carel Faber and Valentijn Karemaker for their valuable comments on those. 

\section{Final types and partitions}
\label{sec:sec2}
Let $g\geq 1$ be an integer. We first introduce the basic notions of the final types and partitions we will use throughout this article. Final types are non-decreasing sequences of integers with a maximum step of $1$ with some initial properties; see below for the precise definition we will use.
We say that $\delta = \delta_1\delta_2\ldots \delta_n$ is a \textbf{partition} of $g$ if $\delta_i \in \Z$ and $$1 \leq \delta_n \leq \delta_{n - 1} \leq \ldots \leq \delta_2 \leq \delta_1 \leq g\quad \text{and} \quad \delta_1 + \delta_2 + \ldots + \delta_n = g. $$
In this case, we write $$\delta \vdash g \quad \text{and} \quad |\delta| = n.$$

\subsection{Final types}
\label{subsec:final_types}
Consider a function $$\nu: \{0, 1, \ldots, g\} \to \{0, 1, \ldots, g\}$$ such that $$\nu(0) = 0\quad \text{and} \quad \nu(i) \leq \nu(i + i) \leq \nu(i) + 1.$$ We
say that $\nu$ is a \textbf{final type} of length $g$ and write $$\nu \vdash g.$$ Taking into account that $\nu$ is determined by the values of $\nu(i)$ with $1 \leq i \leq g$, we also write $$\nu = (\nu(1), \nu(2), \ldots, \nu(g)), $$
and think of $\nu$ as of sequence of $g$ numbers. 
Furthermore, let us denote $$\nu^i(j) = \underbrace{\nu \circ \ldots \circ \nu}_i(j) \quad \text{with} \quad \nu^0(j) = j, \quad \text{for any}\quad 0 \leq i, j \leq g.$$ 
Lastly, we denote $$f(\nu) = \nu^{g}(g);$$ equivalently, $f(\nu)$ is the largest $f \geq 0$ such that $\nu(i) = i$ for every~${0 \leq i \leq f}$. 

\begin{rem}
In total, there are $2^g$ final types $\nu \vdash g$; namely, this count follows from the property that $\nu(i) = \nu(i - 1)$ or $\nu(i) = \nu(i - 1) + 1$ for any $1 \leq i \leq g$.     
\end{rem}

\subsection{Partitions defined by final types}
\label{subsec:partitions_coming_from_final_types}
Given a final type $\nu \vdash g$, let 
\begin{equation}
\delta_{\nu}(i) = \nu^{i - 1}(g) - \nu^{i}(g)\quad \text{for all} \quad 1 \leq i \leq g.
\label{eqn:definition_delta}
\end{equation}
Note that $$\delta_{\nu}(i +1) \leq \delta_{\nu}(i)\quad \text{for all}\quad 1 \leq i < g.$$ If $n$ is the largest number such that $\delta_{\nu}(n) \neq 0$, then $$\delta_{\nu} = \delta_{\nu}(1)\delta_{\nu}(2) \ldots \delta_{\nu}(n)$$
is a partition of $g - f(\nu)$, i.e., $$\delta_{\nu}(1) + \ldots +  \delta_{\nu}(n) = g - f(\nu).$$ 

On the other hand, let $0 \leq f \leq g$ and let $\delta = \delta_1\delta_2 \ldots \delta_n \vdash g - f$ be  a partition with $\delta = |n|$. By the defining property \eqref{eqn:definition_delta}, final type $\nu \vdash g$ satisfies $\delta_{\nu} = \delta$ if and only if 
$$
\nu(f + \sum_{i = j}^n\delta_i) = f + \sum_{i = j + 1}^n\delta_i, \quad \text{for all} \quad j = 1, \ldots, n,
$$
or, in other words, 
\begin{equation}
  \nu =      (\underbrace{1, 2, \ldots, f}_{f}, \underbrace{*, \ldots, *, f}_{\delta_n}, \underbrace{*, \ldots, *, f + \delta_n}_{\delta_{n-1}}, \underbrace{*, \ldots, *, f + \sum_{i = n-1}^n\delta_i}_{\delta_{n-2}}, \ldots,  \underbrace{*, \ldots, *, f + \sum_{i = 2}^n \delta_i}_{\delta_1}),  
\label{eqn:nu_det_by_delta}
\end{equation}
where each entry $*$ can be any integer as long as $\nu(i) \leq \nu({i + 1}) \leq \nu(i) + 1$ for $0\leq i \leq g - 1$ with $\nu(0) = 0$. 

\subsection{Relations between final types and partitions}
The natural question is: given a partition $\delta \vdash g-f$ for some $g \geq 1$ and $0 \leq f \leq g$, what is the number $N(\delta)$ of final types~$\nu \vdash g$ such that $\delta = \delta_{\nu}$. We compute this in the following proposition. 

\begin{prop}
Let $g \geq 1$, $0 \leq f \leq g$ and let $\delta = \delta_1\delta_2\ldots\delta_n \vdash g - f$ be a partition of $g - f$ with $|\delta| = n$.  
Then, there are exactly $$N(\delta) = \prod_{i = 1}^{n-1}{\delta_i \choose \delta_{i + 1}}$$ final types $\nu \vdash g$ such that $\delta = \delta_{\nu}$. 
\label{prop_partitioF_finaltype_identity}
\end{prop}

\begin{proof}
 The number of sequences $(\nu'(1), \ldots, \nu'(k))$ such that $$\nu'(0) = m,\quad \nu'(k) = M,\quad \text{and}\quad \nu'(i - 1) \leq \nu'(i) \leq \nu'(i -1) + 1,\text{ } 1\leq i \leq k$$ equals ${k \choose M - m}$: out of $k$ indices $1 \leq i \leq k$, we choose $M-m$ for which $\nu'(i - 1) < \nu'(i)$.
 
Let $\delta = \delta_1\delta_2\ldots\delta_n \vdash g - f$. 
As we noted in Subsection \ref{subsec:partitions_coming_from_final_types}, the final types $\nu \vdash g$ such that $\delta_{\nu} = \delta$ must be of the form~\eqref{eqn:nu_det_by_delta}. Using the observation of the preceding paragraph, we conclude that the number of final types $\nu \vdash g$ such that $\delta_{\nu} = \delta$ equals the product of ${\delta_{i} \choose \delta_{i + 1}}$ for $1 \leq i < n$. 
\end{proof}

As a nice combinatorial consequence, we find the following identity. 

\begin{cor}
\label{cor:nice_identity_g_g}
For any $g \geq 1$, the following identity holds: $$ \sum_{\delta \vdash g }\prod_{i = 1}^{|\delta| - 1} {\delta_{i} \choose \delta_{i + 1}} = 2^{g - 1}, $$ where the sum is indexed over the set of all partitions $\delta$ of $g$, and $|\delta| = n$ if~${\delta = \delta_1\ldots\delta_n}$. 
\end{cor}

\begin{proof}
    The claimed identity follows by counting the number of distinct final types $\nu \vdash g$ with $f(\nu) = \nu^g(g) = 0$, or equivalently $\nu(1) = 0$, in two different ways.

    First, the number of final types $\nu \vdash g$ with $\nu(1) = 0$ is $2^{g - 1}$; see Subsection~\ref{subsec:final_types}. On the other hand, for every $\nu \vdash g$ with $f(\nu) = 0$, the associated partition $\delta_\nu \vdash g$ is a partition of~$g$. Therefore, the number of such final types equals $$\sum_{\delta \vdash g} N(\delta),$$ where $N(\delta)$ denotes the number of $\nu \vdash g$ satisfying $\delta_\nu = \delta$. By Proposition \ref{prop_partitioF_finaltype_identity}, the claimed identity follows.
\end{proof}

Moreover, we obtain the following criterion of when $N(\delta) = 1$, i.e., when there is a unique final type $\nu \vdash g$ such that $\delta_{\nu} = \delta$. 

\begin{cor}
\label{cor:unique_nu_for_fixed_delta}
Let $g \geq 1$, $0 \leq f < g$, and let $\delta \vdash g - f$. Then, there is a unique final type $\nu \vdash g$ such that $\delta_{\nu} = \delta$ if and only if $\delta = \underbrace{d\ldots d}_e$ for some $d, e \in \Z_{> 0}$ such that $d\cdot e = g - f$. In that case, $$\nu = (\underbrace{1, 2, \ldots, f}_{f}, \underbrace{f, \ldots, f}_{d}, \underbrace{f + 1, f + 2, f + 3 \ldots, f + (e-1)d}_{(e - 1)d}). $$ 
\end{cor}

\begin{proof}
Let $\delta = \delta_1\ldots\delta_n \vdash g$ be a partition of $g$ with $|\delta| = n$. By Proposition \ref{prop_partitioF_finaltype_identity}, we find that~$N(\delta)$, the number of final types $\nu \vdash g$ such that $\delta_{\nu} = \delta$, equals $N(\delta) = \prod_{i = 1}^{n - 1}{\delta_i \choose \delta_{i + 1}}$. In particular, $$N(\delta) = 1\quad \text{if and only if}\quad {\delta_i \choose \delta_{i + 1}} = 1 \text{ for all } 1 \leq i < n.$$ Since $1 \leq \delta_{i + 1} \leq \delta_{i} \leq g$, it follows that $$N(\delta) = 1\quad \text{if and only if}\quad \delta_1 = \delta_2 = \ldots = \delta_n = d$$ for some $d \in \Z_{>0}$ and $e = n \in \Z_{>0}$. 
The description of final types $\nu$ such that $\delta_{\nu} = \underbrace{d \ldots d}_e$ follows from \eqref{eqn:nu_det_by_delta}. 
\end{proof}

\subsection{Motivation for studying these relationships} 
\label{subsec:motivation}
Our motivation for studying the final types and partitions in this section comes from the study of the $p$-torsion $A[p]$ of $g$-dimensional principally polarized abelian varieties $A$ over $k = \Fbar_p$, for a prime number~$p>0$. 

A polarized Dieudonn\'e module is a quadruple $(M, F, V, b)$, where $M$ is a $2g$-dimensional $k$-vector space, equipped with $p$- and $p^{-1}$-linear operators $F$ and $V$, and a non-degenerate alternating bilinear form $b$ which satisfy $\ker(F) = \im(V)$, $\ker(V) = \im(F)$, and $b(F(x), y) = b(x, V(y))^p$ for all $x, y \in M$. 

The set of isomorphism classes of $(M, F, V, b)$ with $\dim_k M = 2g$ is in bijection with the set of all final types $\nu \vdash g$; this was observed in \cite{oort}, where the notion of a final type was introduced. Under this bijection, we have $$\nu^i(g) = \dim_k V^{i+1}(M)$$ for all $0 \leq i \leq g$. Furthermore, this also implies that the set of isomorphism classes of $(M, V)$ is in bijection with the set of partitions $\delta_\nu \vdash g - f(\nu)$, as described in Subsection~\ref{subsec:partitions_coming_from_final_types}. The use of partitions to combinatorially describe these pairs was first introduced in~\cite{hassewitt}.

Given a $g$-dimensional principaly polarized abelian varieties $A$ over $k$, the Dieudonn\'e module $\mathbb{D}(A[p])$ is a polarized Dieudonn\'e module with $\dim_k\mathbb{D}(A[p]) = 2g$, equipped with $F, V,$ and $b$ as above. By the discussion above:
\begin{itemize}
    \item the isomorphism class of $\mathbb{D}(A[p])$ is determined by the final type $\nu \vdash g$ of $A$, denoted $\nu = \nu(A)$ and also called the Ekedahl-Oort type of $A$, while
    \item  the action of $V$ on $\mathbb{D}(A[p])$ is encoded by the partition $\delta_\nu \vdash g - f(\nu)$. 
\end{itemize}
For example, Corollary \ref{cor:unique_nu_for_fixed_delta} gives a criterion when the action of $V$ on $\mathbb{D}(A[p])$ uniquely determines the Ekedahl-Oort type of $A$.

\section{Final types and $n$-nacci numbers}
\label{sec:sec3}
In this section, given $g \geq 1$ and $n \geq 1$, we study the numbers $F_n(g)$ defined by: 
\begin{equation}
F_{n}(g):=\#\{\nu \vdash g \text{ final type}: \nu^n(g) = 0\},
\label{eqn:F_n(g)_definition}
\end{equation}
 Note that $F_{1}(g)$ equals the number of final types $\nu \vdash g$ such that $\nu(g) = 0$. Hence, ~$${F_{1}(g) = 1}$$ as the only such final type is $$\nu = (0, 0, \ldots, 0).$$ Furthermore, note that $F_g(g)$ equals the number of final types $\nu \vdash g$ such that $f(\nu) = \nu^g(g) = 0$; this was discussed in Corollary~\ref{cor:nice_identity_g_g}, and we concluded that $$F_g(g) = 2^{g - 1}.$$ 
More generally, we have the following result. 

\begin{lem}
\label{lem:F_n(g)_defining_identity}
For any $g \geq 1$ and $n \geq 1$, 
\begin{equation}
F_n(g) = \sum_{\substack{0 \leq \delta_n \leq \delta_{r - 1} \leq \ldots \delta_1 \leq g \\ \delta_1 + \ldots + \delta_n = g}} {\delta_1 \choose \delta_2}{\delta_2 \choose \delta_3}\ldots {\delta_{n -1} \choose \delta_{n}}
\label{eqn:F_n(g)_defining_identity}
\end{equation}
\end{lem}
\begin{proof}
Given a final type $\nu \vdash g$, let $\delta_{\nu}$ be a partition defined by $\nu$ as in Subsection \ref{subsec:partitions_coming_from_final_types}. 
By~\eqref{eqn:definition_delta},  $$\nu^n(g) = 0\quad \text{is equivalent to}\quad|\delta_{\nu}| \leq n.$$ The result follows from Proposition \ref{prop_partitioF_finaltype_identity} by summing over all partitions $\delta$ of length $|\delta| = i$ for $1 \leq i \leq n$. 
\end{proof}

Note that, for $n \geq 1$, the right-hand side of \eqref{eqn:F_n(g)_defining_identity} is also defined in the case $g < 1$ and equals  $\begin{cases}
0& \text{if}\; g < 0 \\
1& \text{if}\; g = 0
\end{cases}$. Therefore, let us extend the definition given above of $F_n(g)$ by setting 
\begin{equation}
F_n(g) = \begin{cases}
0& \text{if}\; g < 0 \\
1& \text{if}\; g = 0 \\
F_n(g) \text{ as in }\eqref{eqn:F_n(g)_definition}& \text{if}\; g >0
\end{cases}.
\label{eqn:F_n(g)_definitioF_modified}
\end{equation}

\subsection{On $n$-nacci numbers}
\label{subsection::nbonacci}
Given an integer $n \geq 1$, let $(\Tilde{F}_n(g))_{g \geq 1}$ denote the $n$-nacci sequence, that is, the sequence defined by 
$$\Tilde{F}_n(g) = \Tilde{F}_n(g - 1) + \Tilde{F}_n(g - 2) + \ldots + \Tilde{F}_n(g - n)$$ for $g > n$, 
with initial conditions $$\Tilde{F}_n(1) = \Tilde{F}_n(2) = \ldots = \Tilde{F}_n(n - 1) = 0\quad \text{and}\quad \Tilde{F}_n(n) = 1.$$ 

The case $n = 1$ is trivial, since then $\Tilde{F}_1(g) = 1$ for all $g \geq 1$. 
For $n = 2$, $\Tilde{F}_2$ is the Fibonacci sequence \cite[A000045]{oeis}, which motivates the name. 
Similarly, for $n = 3$, $\Tilde{F}_3$~is the sequence of tribonacci numbers \cite[A000073]{oeis}; for $n = 4$, $\Tilde{F}_4$ is the sequence of tetranacci numbers \cite[A000078]{oeis}; and so on.

\subsection{Relations between $F_n(g)$ and $n$-nacci numbers}

Lastly, we relate the numbers $F_n(g)$ and $\Tilde{F}_n(g)$ for any $n \geq 2$. For $n = 2$, this is a classical result already noted by Lucas in \cite[p.~7]{luc}. Although we believe that the statement below for $n \geq 3$ is known, at least in some cases, we were unable to find a reference, and therefore we present and prove it~here.

\begin{thm} 
Let $g \geq 1$, $n \geq 1$, and let $F_n(g)$ be as in \eqref{eqn:F_n(g)_definitioF_modified}. Then $$F_n(g) = F_n(g - 1) + F_n(g - 2) + \ldots + F_n(g - n).$$
In particular, $F_n(g) = \Tilde{F}_n(g + n)$, where $(\Tilde{F}_n(g))_{g \geq 1}$ is the sequence of $n$-nacci numbers, and Theorem from Section \ref{sec:intro} holds. 
\end{thm}

\begin{proof}
The proof of this theorem relies on elementary combinatorial arguments, in particular on the basic identity for binomial coefficients, Pascal’s rule ${a \choose b} = {a - 1 \choose b} + {a \choose b - 1}$ for any~${a, b \in \Z_{>0}}$. 
By Lemma \ref{lem:F_n(g)_defining_identity}, 
\begin{equation}
F_n(g) = \sum_{\substack{0 \leq \delta_n \leq \ldots \leq  \delta_1 \leq g \\ \delta_1 + \ldots + \delta_n = g}}  {\delta_1 \choose \delta_2}{\delta_2 \choose \delta_3}\ldots {\delta_{n -1} \choose \delta_{n}} = G_1(g) + \ldots + G_n(g),    \label{eqn:F_n(g)}
\end{equation}
where $$G_s(g) =  \sum_{\substack{1 \leq \delta_s \leq \ldots \leq  \delta_1 \leq g \\ \delta_1 + \ldots + \delta_s = g}}  {\delta_1 \choose \delta_2}{\delta_2 \choose \delta_3}\ldots {\delta_{s -1} \choose \delta_{s}}.$$

Applying Pascal's rule, we get the following for $s \geq 2$: 
\begin{align*}
    G_s(g) &= \sum_{\substack{1 \leq \delta_s \leq \delta_{s - 1} \leq \ldots \delta_1 \leq g \\ \delta_1 + \ldots + \delta_s = g}} \left ( {\delta_1 - 1 \choose \delta_2} + {\delta_1 - 1 \choose \delta_2 - 1} \right ){\delta_2 \choose \delta_3}\ldots {\delta_{s -1} \choose \delta_{s}} \\ 
    &= \sum_{\substack{1 \leq \delta_s \leq \delta_{s - 1} \leq \ldots \delta_1' \leq g - 1 \\ \delta_1' + \ldots + \delta_s = g - 1 \\ \delta_1' = \delta_1 - 1}} {\delta_1' \choose \delta_2} {\delta_2 \choose \delta_3}\ldots {\delta_{s -1} \choose \delta_{s}} + 
   \sum_{\substack{1 \leq \delta_s \leq \delta_{s - 1} \leq \ldots \delta_1' \leq g - 1 \\ \delta_1' + \ldots + \delta_s = g - 1 \\ \delta_1' = \delta_1 - 1}} {\delta_1' \choose \delta_2 - 1} {\delta_2 \choose \delta_3}\ldots {\delta_{s -1} \choose \delta_{s}} \\
    &= G_s(g - 1) + 
   \sum_{\substack{1 \leq \delta_s \leq \delta_{s - 1} \leq \ldots \delta_1' \leq g - 2 \\ \delta_1' + \ldots + \delta_s = g - 2 \\ \delta_i' = \delta_i - 1, 1 \leq i \leq 2}} {\delta_1' \choose \delta_2'}  {\delta_2' + 1 \choose \delta_3} \ldots {\delta_{s -1} \choose \delta_{s}}\\
    &= G_s(g - 1) + 
   \sum_{\substack{1 \leq \delta_s \leq \delta_{s - 1} \leq \ldots \delta_1' \leq g - 2 \\ \delta_1' + \ldots + \delta_s = g - 2 \\ \delta_i' = \delta_i - 1, 1 \leq i \leq 2}} {\delta_1' \choose \delta_2'} \left ( {\delta_2' \choose \delta_3} + {\delta_2' \choose \delta_3 - 1} \right )\ldots {\delta_{s -1} \choose \delta_{s}}\\
   &= G_s(g - 1) + 
   \sum_{\substack{1 \leq \delta_s \leq \ldots \leq \delta_1' \leq g - 2 \\ \delta_1' + \ldots + \delta_s = g - 2 \\ \delta_i' = \delta_i - 1, 1\leq i \leq 2}} {\delta_1' \choose \delta_2'} {\delta_2' \choose \delta_3} \ldots {\delta_{s -1} \choose \delta_{s}} +    \sum_{\substack{1 \leq \delta_s \leq \ldots \leq \delta_1' \leq g - 2 \\ \delta_1' + \ldots + \delta_s = g - 2 \\ \delta_i' = \delta_i - 1, 1\leq i \leq 2}} {\delta_1' \choose \delta_2'} {\delta_2' \choose \delta_3 - 1} \ldots {\delta_{s -1} \choose \delta_{s}} \\
   &= G_s(g - 1) + G_s(g - 2) + \sum_{\substack{1 \leq \delta_s \leq \ldots \leq \delta_1' \leq g - 3 \\ \delta_1' + \ldots + \delta_s = g - 3 \\ \delta_i' = \delta_i - 1, 1\leq i \leq 3}} {\delta_1' \choose \delta_2'} {\delta_2' \choose \delta'_3} {\delta_3' + 1 \choose \delta_4} \ldots {\delta_{s -1} \choose \delta_{s}}\\
   &= G_s(g - 1) + \ldots + G_s(g - s + 2) + \sum_{\substack{1 \leq \delta_s \leq \ldots \leq \delta_1' \leq g - s + 1 \\ \delta_1' + \ldots + \delta_s = g - s +1 \\ \delta_i' = \delta_i - 1, 1\leq i \leq s - 1}} {\delta_1' \choose \delta_2'} \ldots {\delta_{s -2}' \choose \delta_{s - 1}'} {\delta_{s -1}' + 1 \choose \delta_{s}}.
\end{align*}
Finally, by writing ${\delta_{s -1}' + 1 \choose \delta_{s}} = {\delta_{s -1}' \choose \delta_{s}} + {\delta_{s -1}' \choose \delta_{s} - 1}$ and $\delta'_s = \delta_s-1$, we obtain 
\begin{align*}
G_s(g) &= G_s(g - 1) + \ldots + G_s(g - s + 1) + \sum_{\substack{0 \leq \delta_s' \leq \ldots \leq \delta_1' \leq g - s \\ \delta_1' + \ldots + \delta_s' = g - s}} {\delta_1' \choose \delta_2'} \ldots {\delta_{s -2}' \choose \delta_{s - 1}'} {\delta_{s -1}' \choose \delta_{s}'}  \\
&= G_s(g - 1) + \ldots + G_s(g - s + 1) + F_s(g - s).
\end{align*}
Therefore, for any $g \geq 1$ and $s \geq 2$, it follows that: \begin{equation}
    G_s(g) = \sum_{i = 1}^{s - 1}G_s(g - i) + \sum_{i = 1}^sG_i(g - s),
\label{eqn:M_s(g)}
\end{equation}
where we used $F_s(g - s) = \sum_{i = 1}^sG_i(g - s)$ as observed in \eqref{eqn:F_n(g)}. 
Moreover, note that \eqref{eqn:M_s(g)} also holds for $g \geq 1$ and $s = 1$ since $G_1(g) = G_1(g - 1) = 1$.

Using \eqref{eqn:M_s(g)} for $g \geq 1$ and $s \geq 1$, we obtain that 
\begin{equation*}
F_n(g) = \sum_{s = 1}^n\left ( \sum_{i = 1}^{s - 1}G_s(g - i) + \sum_{i = 1}^sG_i(g - s) \right)
\end{equation*}
Note that on the right-hand side of this equality, every element $G_i(j)$ for $1 \leq i \leq n$ and $g - n \leq j \leq g - 1$ and appears exactly once. Therefore, by rearranging the sum, we obtain that
$$F_n(g) = \sum_{j = g - n}^{g - 1} \sum_{i = 1}^{n}G_i(j) = \sum_{j = g - n}^{g - 1} F_n(j) = F_n(g - 1) + F_n(g - 2) + \ldots + F_n(g - n).$$
This proves the desired recursive relations. Lastly, this identity and the initial conditions in \eqref{eqn:F_n(g)_definitioF_modified} for $F_n(g)$ with $g \leq 0$ imply that $$(F_n(g))_{g \geq 1} = (\Tilde{F}_n(g + n))_{g \geq 1},$$ where $(\Tilde{F}_n(g))_{g \geq 1}$ is as in Subsection~\ref{subsection::nbonacci}, 
i.e., that $(F_n(g))_{g \geq 1}$ is the $n$-nacci sequence shifted by~$n$. 
\end{proof}

\section{Comparing three partial orders of partitions}
\label{sec:order}

We recall some basic notions from Section \ref{sec:sec2}.
First, recall that $\delta = \delta_1\delta_2\ldots\delta_n \vdash g$ denotes a partition of $g$ into $|\delta| = n$ parts, i.e., that $\delta_1 \geq \delta_2 \geq \ldots \geq \delta_n > 0$ and $\delta_1 + \ldots + \delta_n = g$. Furthermore, given a final type $\nu \vdash g$, we denote by $\delta_{\nu}$ the partition defined by \eqref{eqn:definition_delta}. Lastly, given a partition $\delta \vdash g$, any final type $\nu \vdash g$ such that $\delta = \delta_{\nu}$ satisfies \eqref{eqn:nu_det_by_delta}.

In the preceding section, we studied the relationship between final types and partitions, which led us to some interesting combinatorial consequences. 
As we explained in Subsection~\ref{subsec:motivation}, our motivation for studying these objects comes from the study of the $p$-torsion of abelian varieties in characteristic $p$. 
The same motivation leads us to consider the lexicographical ordering on the set of all final types $\nu \vdash g$ for a fixed $g \geq 2$, when we say that $$\nu = (\nu(1), \ldots, \nu(g)) \leq \nu' = (\nu'(1), \ldots, \nu'(g)) \text{ if and only if } \nu(i) \leq \nu'(i) \text{ for all }1 \leq i \leq g.$$
Given $g \geq 2$, this induces a partial ordering on the set of all partitions $\delta \vdash g$ defined by $$\delta \leq_{ft} \delta' \text{ if and only if for all } \nu\vdash g \text{ with } \delta = \delta_{\nu} \text{ there is } \nu' \vdash g \text{ such that } \delta' = \delta'_{\nu'} \text{ and } \nu \leq \nu'.$$
(The notation $\leq_{ft}$ indicates that this relation is induced by the partial order on final types.) Our wish is to compare this order with two natural and well-known orders on the set of all partitions of a fixed integer. 

A natural partial ordering on the set of all partitions $\delta \vdash g$ is given by $$\delta \leq_{pp} \delta' \text{ if and only if } \delta \text{ is obtained by grouping some parts of } \delta', $$
i.e., if $\delta = \delta_1\ldots\delta_m$, $\delta' = \delta'_1\ldots\delta'_n$, where, for each $i$, $\delta_i = \sum \delta'_{j_i}$ for some $j_i$ and the multisets $$\{\delta_1', \ldots, \delta_n'\}\quad \text{ and } \quad \{\delta_{j_i}': 1 \leq i \leq m\}$$ coincide. (The notation $\leq_{pp}$ indicates that $\delta'$ is a "partial partition" of $\delta$.) We illustrate some examples when $\delta \leq_{pp} \delta'$ below. 

\begin{exmp}
\label{exmp:345}
Below is the description of the ordering $\leq_{pp}$ in the cases $g = 3, 4, 5$.

 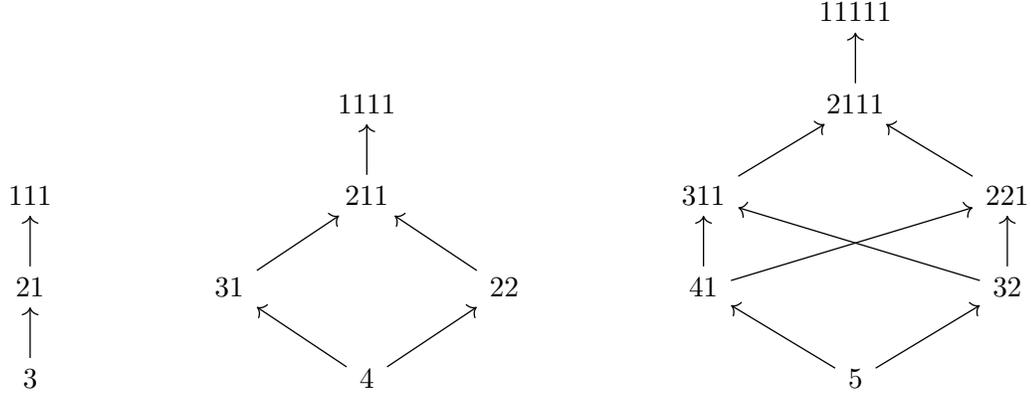
\begin{figure}[H]
 \begin{center}
\begin{tikzcd}
	&&&&&&& 11111 \\
	&&& 1111 &&&& 2111 \\
	111 &&& 211 &&& 311 && 221 \\
	21 && 31 && 22 && 41 && 32 \\
	3 &&& 4 &&&& 5
	\arrow[from=5-1, to=4-1]
	\arrow[from=4-1, to=3-1]
	\arrow[from=5-4, to=4-5]
    \arrow[from=4-3, to=3-4]
	\arrow[from=3-4, to=2-4]
	\arrow[from=4-5, to=3-4]
	\arrow[from=5-4, to=4-3]
	\arrow[from=5-8, to=4-7]
	\arrow[from=5-8, to=4-9]
	\arrow[from=4-7, to=3-7]
	\arrow[from=4-9, to=3-9]
	\arrow[from=4-9, to=3-7]
	\arrow[from=4-7, to=3-9]
	\arrow[from=3-7, to=2-8]
	\arrow[from=3-9, to=2-8]
	\arrow[from=2-8, to=1-8]
\end{tikzcd}
\end{center}
     
     \caption{Partitions $\delta \vdash g$, for $g \in \{3, 4, 5\}$; $\delta \rightarrow \delta'$ if ${\delta} \leq_{pp} {\delta'}$.}
     \label{fig:partitions_genus345}
 \end{figure} 

\noindent The transitivity of  $\leq_{pp}$ is reflected in this example by observing, for instance, that $3 \leq_{pp} 111$, $22 \leq_{pp} 1111$, and $32\leq_{pp} 2111$. 
\end{exmp}

\begin{rem}
Following the definition, one easily verifies that for partitions $\delta \vdash g$ and $\delta' \vdash g$ with $g \in \{3, 4, 5\}$, the diagrams in Figure \ref{fig:partitions_genus345} satisfy
$\delta \rightarrow \delta'$ whenever ${\delta} \leq_{ft} {\delta'}$. (Alternatively, this follows from Corollary \ref{cor:lastdeltaisone} below.)
This raises the question of how the two orders $\leq_{ft}$ and~$\leq_{pp}$ compare.     
\end{rem}

There is another natural order on the set of all partitions $\delta \vdash g$ for a fixed $g\geq 2$, known as the \emph{dominance order}. In fact, we consider two of its variants, right-dominance order $\leq_{do}$ and left-dominance order $\leq_{do'}$, defined by
$$\delta = \delta_1\ldots\delta_m \leq_{do} \delta' = \delta'_1\ldots\delta'_n \text{ if and only if } \sum_{i = |\delta'| - k}^{|\delta'|} \delta_{i}' \leq \sum_{i = |\delta| - k}^{|\delta|} \delta_{i} \text{ for all } k \geq 0, $$
where $\delta_{i} = \delta'_{i} = 0$ for $i \leq 0$, and
$$\delta = \delta_1\ldots\delta_m \leq_{do'} \delta' = \delta'_1\ldots\delta'_n \text{ if and only if }\sum_{i = 1}^k \delta_{i} \leq \sum_{i = 1}^k \delta_{i}' \text{ for all } k \geq 0.$$
Note that for $\delta, \delta' \vdash g$ such that $|\delta| = |\delta'|$ it holds that $\delta \leq_{do} \delta'$ if and only if $\delta \leq_{do'} \delta'$. 
\begin{rem}
\label{rem:do}
For example, for two partitions $31\vdash 4, 22 \vdash 4$, we see that $22 \leq_{do} 31$ and $22 \leq_{do'} 31$. On the other hand, note that $22 \not\leq_{*} 31$ and $31 \not\leq_{*} 22$ for $* \in \{\text{ft}, \text{pp}\}$. Also, note that $4 \leq_{do} 31$ while $31 \leq_{do'} 4$ and~$4 \not \leq_{do'} 31$.
This shows that the two orders dominance orders $\leq_{do}$ and $\leq_{do'}$ differ from $\leq_{ft}$ and $\leq_{pp}$ in general, and that $\leq_{do}$ and $\leq_{do'}$ are themselves distinct. Lastly, as illustrated in Figure \ref{fig:partitions_genus345}, note that either of the relations $\delta \leq_{pp} \delta'$ or~$\delta \leq_{ft} \delta'$ implies~$\delta \leq_{do} \delta'$ when $\delta, \delta' \vdash g$ with $g \in \{3, 4, 5\}$. 
\end{rem}

\subsection{Comparing $\leq_{ft}$, $\leq_{pp}$, and $\leq_{do}$}

Let $g \geq 2$ be a fixed integer throughout this section. Motivated by the preceding remark, in this subsection, we focus on the relationships between the orders $\leq_{ft}$, $\leq_{pp}$, and $\leq_{do}$. 
We begin with the first lemma, which will be used throughout this section, as it allows us in some cases to reduce the comparison of partitions of $g$ to partitions of $g' < g$.

\begin{lem}
Let $\delta = \delta_{1}\ldots\delta_m \vdash g$ and $\delta' = \delta_{1}'\ldots\delta_n' \vdash g$ be two partitions of $g$, and let $h \geq \delta_1$ and $f \leq \min(\delta_m, \delta_n')$ be positive integers. 
Then ${\delta} \leq_{ft} {\delta}'$ if and only if ${\Tilde{\delta}} \leq_{ft} {\Tilde{\delta'}}$, where $\Tilde{\delta} = h\delta_{1}\ldots\delta_m f$ and $\Tilde{\delta'} = h\delta_{1}'\ldots\delta_n'f$. 
\label{lemma:extension_delta}
\end{lem} 

\begin{proof}
Let $\nu \vdash g$ (resp. $\Tilde{\nu} \vdash g$) be a final type such that $\delta = \delta_{\nu}$ (resp. $\delta' = \delta'_{\nu'}$). By \eqref{eqn:nu_det_by_delta}, $$\Tilde{\nu}(i) = 0, 1\leq i \leq f, \quad  \Tilde{\nu}(f + i) = f + \nu(i), 1\leq i \leq g,  \text{ and}\quad \Tilde{\nu}(f + g + h) = f + g.$$ A similar description holds for $\Tilde{\nu}'$. It follows from the definition that ${\delta} \leq_{ft} {\delta}'$ if and only if~${\Tilde{\delta}} \leq_{ft} {\Tilde{\delta'}}$.
\end{proof}

Having established the preceding lemma, which allows us to compare two partitions by induction, we now present a first comparison of $\leq_{ft}$ and $\leq_{pp}$ in the case where one partition is obtained from the other by grouping only two of its parts. Specifically, we have the following result.

\begin{lem}
Let $\delta = \delta_1\ldots\delta_{n + 1} \vdash g$ be a partition. Assume that $\delta' = \delta_1'\ldots\delta_n' \vdash g$ is a partition obtained by grouping two parts of $\delta$, i.e., such that multisets $\{\delta_1', \ldots, \delta_n'\}$ and $\{\delta_i + \delta_j\} \cup \{\delta_k: k \neq i, j\}$ coincide for some $i \neq j$. Then ${\delta'}\leq_{ft} {\delta}$.
\label{lem:grouping}
\end{lem}

\begin{proof} By Lemma \ref{lemma:extension_delta} and induction on $g$ and $n$, we may assume that $\delta_1' = \delta_1 + \delta_{n+1}$ and $\delta_i = \delta_i'$, for $2\leq i \leq n$. We need to show that for any $\nu'$ such that $\delta' = \delta'_{\nu'}$, there is a $\nu$ such that $\delta = \delta_{\nu}$ such that $\nu' \leq \nu$, i.e., that 
$\nu'(j) \leq \nu(j)$ for all $1\leq j \leq g$. 

For any final type $\Tilde{\nu}$, the following inequality holds: $\Tilde{\nu}(a + b) \leq \Tilde{\nu}(a) + b$, where $a, b \geq 0$; this follows directly from the definition. By \eqref{eqn:nu_det_by_delta}, for any $\nu'$ such that $\delta' = \delta'_{\nu'}$ and $\nu$ such that $\delta = \delta_{\nu}$, we thus have 
\begin{equation}
\nu'(\sum_{i = j}^{n}\delta_i) = \nu'(\sum_{i = j}^n\delta_i') = \sum_{i = j + 1}^{n}\delta_i = \nu(\delta_{n + 1} + \sum_{i = j}^{n}\delta_i) - \delta_{n + 1} \leq \nu(\sum_{i = j}^{n}\delta_i) = \nu(\sum_{i = j}^{n}\delta_i') \text{ for any }2\leq j \leq n.
\label{eqn:ft_comparison_1}
\end{equation}
Similarly, 
\begin{equation}
\nu'(\sum_{i = 1}^{n + 1}\delta_i) = \nu'(\sum_{i = 1}^n\delta_i') = \sum_{i = 2}^n\delta_i' = \sum_{i = 2}^n\delta_i \leq \delta_{n + 1} + \sum_{i = 2}^n\delta_i  = \nu(\sum_{i = 1}^{n + 1}\delta_i) = \nu(\sum_{i = 1}^{n}\delta_i').
\label{eqn:ft_comparison_2}
\end{equation}
We conclude that for any $\nu$ (resp.~$\nu'$) such that $\delta = \delta_{\nu}$ (resp.~$\delta' = \delta'_{\nu'}$), inequalities \eqref{eqn:ft_comparison_1} and~\eqref{eqn:ft_comparison_2} hold. Furthermore, all such $\nu\vdash g$ (resp.~$\nu'\vdash g$) can be any final types whose values appearing on the right-hand (resp.~left-hand) side of \eqref{eqn:ft_comparison_1} and \eqref{eqn:ft_comparison_2} are prescribed. 
Once again, recall that $\sum_{i = 1}^{n+1}\delta_i = \sum_{i = 1}^n\delta'_i$ and $\sum_{i = j}^n\delta_i = \sum_{i = j}^n\delta'_i$ for any $2 \leq j \leq n$.
Therefore, given any $\nu' \vdash g$ such that $\delta' = \delta'_{\nu'}$, we can always find $\nu \vdash g$ such that $\delta = \delta_{\nu}$ for which~$\nu' \leq \nu$.
\end{proof}

As a consequence of the preceding observations, we can now prove the following proposition.

\begin{prop}
\label{prop:pp_implies_ft}
For any two partitions $\delta \vdash g$ and $\delta' \vdash g$, $\delta \leq_{pp} \delta'$ implies $\delta \leq_{ft} \delta'$.
\end{prop}
\begin{proof}
This follows from Lemma \ref{lem:grouping} and the transitivity of $\leq_{pp}$. 
\end{proof}

In the remainder of this section, we investigate whether/when the reverse implication in the preceding proposition holds. We begin with the following lemma, which compares two partitions having the same number of parts.
 
\begin{lem}
Let $\delta = \delta_1\ldots\delta_n \vdash g$ and $\delta' = \delta_1'\ldots\delta_n' \vdash g$ be two distinct partitions with the same number of parts $|\delta| = |\delta'| = n \geq 1$. Then ${\delta} \not \leq_{ft} {\delta'}$ and ${\delta'} \not \leq_{ft} {\delta}$.
\label{lem:sameparts}
\end{lem}

\begin{proof}
For $g = 2, 3, 4$, this is clear and was observed in Example \ref{exmp:345}. Without loss of generality and by Lemma \ref{lemma:extension_delta} and induction on $g$, we may assume that $\delta_1>\delta_1'$ and $\delta_n \neq \delta_n'$. 
Note the first condition is equivalent to $(*)$: $$\sum_{i = 2}^n\delta_i<\sum_{i = 2}^n\delta_i';$$ alternatively, $\nu(g) < \nu'(g)$ for any $\nu$ (resp.~$\nu'$) such that $\delta = \delta_{\nu}$ (resp.~$\delta' = \delta'_{\nu'}$). This implies $\delta' \not \leq_{ft} \delta$. 
It remains to show that there is a final type $\nu \vdash g$ with $\delta = \delta_{\nu}$ such that for all final types $\nu' \vdash g$ with $\delta' = \delta'_{\nu'}$ it holds that $\nu(m)>\nu'(m)$ for some $m = m(\nu')$.

By $(*)$, there is the largest $2\leq j\leq n$ such that $(\#): \sum_{i = j}^n\delta_i<\sum_{i = j}^n\delta_i'$. 
Consider the type $\nu\vdash g$ such that $\delta = \delta_{\nu}$ and $$\nu(\sum_{i = j}^n\delta_i + 1) = 1 + \sum_{i = j+1}^n\delta_i;$$ the existence of such $\nu$ follows from~\eqref{eqn:nu_det_by_delta} and $j \geq 2$.
The definition of $j$ implies that $$\nu'(\sum_{i = j}^n\delta_i + 1)\underset{(\#)}{\leq} \nu'(\sum_{i = j}^n\delta_i') = \sum_{i = j+1}^n\delta_i' < 1 + \sum_{i = j+1}^n\delta_i = \nu(\sum_{i = j}^n\delta_i + 1). $$ Therefore, we have found $m = \sum_{i = j}^n\delta_i + 1$ for which $\nu(m) > \nu'(m)$ and thus $\delta \not \leq_{ft} \delta'$. The result follows.
\end{proof}

In the proposition below, we show that $\delta \leq_{ft} \delta'$ implies $\delta' \leq_{do} \delta$. 

\begin{prop}
\label{prop:ft_do}
For any two partitions $\delta = \delta_{1}\ldots\delta_{n}\vdash g$ and $\delta' = \delta_1'\ldots \delta_m'\vdash g$, $\delta \leq_{ft} \delta'$ implies $n \leq m$ and $\delta' \leq_{do} \delta$, i.e., $\sum_{i = m - k}^{m}\delta'_i \leq \sum_{i = n - k}^{n}\delta_i$ for any $k \geq 0$. 
\end{prop}
\begin{proof}
As before, using Lemma \ref{lemma:extension_delta}, it suffices to consider the case $\delta_1 \neq \delta_1'$ and $\delta_n \neq \delta_m'$. Note that $\sum_{i = m - k}^{m}\delta'_i \leq \sum_{i = n - k}^{n}\delta_i$ for $k \geq 0$ implies $n \leq m$. (If $n > m$, then the choice $k = m - 1$ would give $g = \sum_{i = 1}^m\delta_i' \leq \sum_{i = n - m + 1}^n\delta_i \leq g - \delta_{n - m} < g$, which is impossible.) 
We want to show 
\begin{equation}
\sum_{i = j + m - n}^{m}\delta'_i \leq \sum_{i = j}^{n}\delta_i \quad \text{for all}\quad 1 \leq j \leq n.
\label{eqn:do_dir_ineq}    
\end{equation}
Note that the preceding inequality trivially holds for $j = 1$.
Furthermore, it holds for $j = n$:
consider any final type $\nu \vdash g$ such that $\delta = \delta_{\nu}$ and $$\nu(\delta_n + 1) = 1$$ (see \eqref{eqn:nu_det_by_delta}), and let $\nu' \vdash g$ be a final type such that $\delta' = \delta'_{\nu'}$ and $\nu \leq_{ft} \nu'$. (This $\nu'$ exists because of $\delta \leq_{ft} \delta'$.) Then, the assumption $\delta_m' > \delta_n$ would imply $$1 = \nu(\delta_n + 1)\leq \nu'(\delta_n + 1) \leq \nu'(\delta_m') = 0.$$ As this is impossible, we conclude $\delta_m'\leq \delta_n$. 

Now, we want to show the inequality \eqref{eqn:do_dir_ineq} for the remaining $2 \leq j < n$. Assume the contrary, and let $2\leq j < n$ be the largest one for which \eqref{eqn:do_dir_ineq} does not hold. Then $$\sum_{i = j  + m - n}^{m}\delta_i' > \sum_{i = j}^{n}\delta_i, \text{ while } \sum_{i = j  + m - n + 1}^{m}\delta_i' \leq \sum_{i = j + 1}^{n}\delta_i$$
Let $\nu \vdash g$ be a final type such that $\delta = \delta_{\nu}$ and $$\nu(\sum_{i = j}^{n}\delta_i + 1) = \sum_{i = j + 1}^{n}\delta_i + 1;$$ see \eqref{eqn:nu_det_by_delta}. Then $\nu \leq_{ft} \nu'$ would imply $$\sum_{i = j + m -n + 1}^{m}\delta_i'< \sum_{i = j + 1}^{n}\delta_i + 1 = \nu(\sum_{i = j}^{n}\delta_i + 1) \leq \nu(\sum_{i = j + m - n}^{m}\delta_i') \leq \nu'(\sum_{i = j + m - n}^{m}\delta_i') = \sum_{i = j + m - n + 1}^{n + 1}\delta_i'.$$
Since this is impossible, \eqref{eqn:do_dir_ineq} holds for every $1 \leq j \leq n$. 
\end{proof}

We present a criterion for comparing two partitions $\delta\vdash g$ and~$\delta' \vdash g$ whose numbers of parts differ by one.

\begin{prop}
Let $\delta = \delta_1\ldots\delta_n \vdash g$ and $\delta' = \delta'_1\ldots\delta'_n\delta_{n+1}'\vdash g$ such that $\delta_1 \neq \delta_1'$ and $\delta_{n}\neq \delta_{n+1}'$. Then ${\delta'}\not\leq_{ft} {\delta}$, while ${\delta}\leq_{ft} {\delta'}$ implies $$\delta_1' < \delta_1 \quad \text{ and } \quad \sum_{i = j + 1}^{n + 1}\delta_i' \leq  \sum_{i = j }^n\delta_i \leq \sum_{i = j}^{n + 1}\delta'_i \text{ for all }1 \leq j \leq n.$$  

\label{prop:comparingdeltas}
\end{prop}

\begin{proof}
First, note that ${\delta}' \not\leq_{ft} {\delta}$ and $\delta \leq_{ft} \delta' \implies \sum_{i = j}^n\delta_i \geq \sum_{i = j + 1}^{n + 1}\delta'_i \text{ for } 2 \leq j \leq n$ follows from Proposition \ref{prop:ft_do}.
Now, let us assume that ${\delta}\leq_{ft} {\delta'}$ and given $\nu \vdash g$ such that $\delta = \delta_{\nu}$, let $\nu' \vdash g$ be a final type such that $\delta' = \delta'_{\nu'}$ and $\nu \leq \nu'$. Since $g - \delta_1 = \nu(g)\leq \nu'(g) = g - \delta_1'$ and $\delta_1 \neq \delta_1'$, we conclude $\delta_1' < \delta_1$, or equivalently, $\sum_{i = 2}^{n}\delta_i < \sum_{i = 2}^{n+1}\delta_i'$. Note that $\nu \leq \nu'$ and $\sum_{i = j}^{n}\delta_i \leq \sum_{i = j}^{n +1}\delta_i'$ imply 
$$\sum_{i = j + 1}^{n}\delta_i = \nu(\sum_{i = j }^{n}\delta_i) \leq \nu'(\sum_{i = j}^{n}\delta_i) \leq \nu'(\sum_{i = j}^{n+1}\delta_i') = \sum_{i = j +1}^{n+1}\delta_i',$$ 
for any $1 \leq j \leq n-1$. 
Therefore, the induction on $j$ implies that
\begin{equation}
\sum_{i = j}^{n}\delta_i \leq \sum_{i = j}^{n +1}\delta_i' \quad \text{for all}\quad 1 \leq j \leq n.
\label{eqn:one_dir_ineq}    
\end{equation}
\end{proof}

As a consequence, we show below that ${\delta}\leq_{ft} {\delta}'$ if and only if $\delta_{ft} \leq_{pp} \delta'$ in a special case. 

\begin{cor}
Let $\delta = \delta_1\ldots\delta_n\vdash g$ and $\delta' = \delta'_1\ldots\delta'_n\delta_{n+1}'\vdash g$ be two partitions such that~$\delta_1 \neq \delta_1'$ and $\delta_n \neq \delta_{n+1}' = 1$. Then ${\delta}\leq_{ft} {\delta}'$ if and only if $\delta \leq_{pp} \delta'$. 
\label{cor:lastdeltaisone}
\end{cor}

\begin{proof}
By Proposition \ref{prop:pp_implies_ft}, ${\delta}\leq_{pp} {\delta}'$ implies $\delta \leq_{ft} \delta'$.
Assume now that ${\delta} \leq_{ft} {\delta}'$.  Applying the obtained inequality \eqref{eqn:one_dir_ineq} from
Proposition~\ref{prop:comparingdeltas} to $\delta$ and $\delta'$ yields $2 \leq \delta_n \leq \delta_n' + 1$ and more generally, $\sum_{i = j}^n \delta_i \leq 1 + \sum_{i = j}^n \delta_i'$ for any $1 \leq j \leq n$. 

First, we claim that either $$\delta_n = \delta_n'\quad\text{ or }\quad \delta_n = \delta_n' + 1.$$ Otherwise $2 \leq \delta_n \leq \delta_n' - 1$, and let us consider $\nu \vdash g$ such that $\delta = \delta_{\nu}$ and $$\nu(\delta_n + 2) = 2$$ (and thus $\nu(\delta_n + 1) = 1$); recall \eqref{eqn:nu_det_by_delta} and note that $\delta_{n - 1} \geq \delta_n \geq 2$ and $\Tilde{\nu}(\delta_n + \delta_{n - 1}) = \delta_n \geq 2$ for any $\Tilde{\nu} \vdash g$ such that $\delta = \delta_{\Tilde{\nu}}$. However, then for any $\nu'\vdash g$ such that $\delta' = \delta'_{\nu'}$ and $\nu \leq \nu'$ it would follow $$2\leq \nu(\delta_n + 2) \leq \nu'(\delta_n' + 1) \leq \nu'(\delta_n' + 1) = 1, $$
which is absurd. Therefore, it holds either $\delta_n = \delta_n'$ or $\delta_n = \delta_n' + 1$. 

Assume that $\delta_n = \delta_n' + 1$. Then \eqref{eqn:one_dir_ineq} for $i = n - 1$ implies $\delta_{n-1} \leq \delta_{n - 1}'$, and by induction we obtain $\delta_{i} \leq \delta_{i}'$ for all $1 \leq i < n$. Since $\delta$ and $\delta'$ are both partitions of $g$, this forces $\delta_{i} = \delta_{i}'$ for all~$1 \leq i < n$ and thus~$\delta \leq_{pp} \delta'$. 

In the remaining case, $\delta_n = \delta_n'$, \eqref{eqn:one_dir_ineq} for $i = n - 1$ implies $\delta_{n-1} \leq \delta_{n - 1}' + 1$, and therefore, either $$\delta_{n-1} = \delta_{n - 1}'\quad \text{ or } \quad\delta_{n-1} \leq \delta_{n - 1}' + 1,$$ using a similar argument as for $i = n$. More generally, the induction on $i$ with the base step~$i = n$ yields an $1 \leq i = j \leq n$ for which $\delta_{j} = \delta_{j}' + 1$, while for all other $i \neq j$ it holds that~$\delta_i = \delta_i'$. This shows that  $\delta \leq_{pp} \delta'$ and finishes the proof.
 \end{proof}

In particular, the preceding corollary (together with Lemma \ref{lemma:extension_delta}) confirms our claim stated in the paragraph preceding Remark \ref{rem:do}, namely, that the diagram in Figure \ref{fig:partitions_genus345} also satisfies $\delta \to \delta'$ if $\delta \leq_{ft} \delta'$. Therefore, for $g \leq 5$, the two orders $\leq_{ft}$ and $\leq_{pp}$ on the set of all partitions $\delta \vdash g$ coincide.
However, we present below an example showing that the assumption $\delta_{n + 1}' = 1$ used in Corollary \ref{cor:lastdeltaisone} was necessary, and that $\delta \leq_{ft} \delta'$ is not equivalent to $\delta \leq_{pp} \delta'$ in general. 

\begin{exmp}
\label{exmp:Example6}
When $\delta_{n+1}'>1$, the result analogous to that in Corollary \ref{cor:lastdeltaisone} no longer holds. The partitions $\delta = 33$ and $\delta' = 222$ provide the first (i.e., smallest $g$) such example. Namely, by \eqref{eqn:nu_det_by_delta}, final types $\nu$ (resp.~$\nu'$) such that $\delta = \delta_{\nu}$ (resp.~$\delta' = \delta'_{\nu'}$) are of the form $\nu = (*, *, 0, *, *, 3)$ (resp.~$\nu' = (*,0, *, 2, *, 4)$). A closer look, following the definition of final types, shows that in both cases there is precisely one final type: $$\nu = (0, 0, 0, 1, 2, 3)\quad \text{and}\quad \nu' = (0, 0, 1, 2, 3, 4),$$ and that $\nu \leq \nu'$.  
 In particular, this shows ${\delta}\leq_{ft} {\delta'}$, while ${\delta}\not\leq_{pp} {\delta'}$. 

 Below is the diagram of all partitions $\delta \vdash 6$, where the relations $\delta \leq_{ft} \delta'$ are indicated. 
  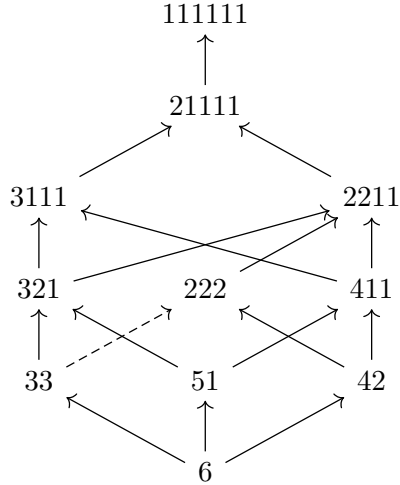
\begin{figure}[H]
 \begin{center}
 \begin{tikzcd}
	& 111111 \\
	& 21111 \\
	3111 && 2211 \\
	321 & 222 & 411 \\
	33 & 51 & 42 \\
	& 6
	\arrow[from=2-2, to=1-2]
	\arrow[from=3-1, to=2-2]
	\arrow[from=3-3, to=2-2]
	\arrow[from=4-1, to=3-1]
	\arrow[from=4-3, to=3-1]
	\arrow[from=4-3, to=3-3]
	\arrow[from=4-2, to=3-3]
	\arrow[from=5-1, to=4-1]
	\arrow[from=5-2, to=4-1]
	\arrow[from=5-2, to=4-3]
	\arrow[from=5-3, to=4-3]
	\arrow[from=5-3, to=4-2]
	\arrow[from=6-2, to=5-1]
	\arrow[from=6-2, to=5-2]
	\arrow[from=6-2, to=5-3]
	\arrow[from=4-1, to=3-3]
	\arrow[dashed, from=5-1, to=4-2]
\end{tikzcd}
\end{center}
     \caption{Partitions $\delta$ of $g = 6$; $\delta \rightarrow \delta'$ if ${\delta} \leq_{ft} {\delta'}$.}
     \label{fig:partitions_genus6}
 \end{figure}
\end{exmp}

\subsection{A summary of the comparison of three partial orders}

 We conclude this section with the summary presented in the following proposition showing an interesting nature of~$\leq_{ft}$. 

\begin{prop}
\label{prop:last}
For any two partitions $\delta \vdash g$ and $\delta' \vdash g$, it holds that $$\delta \leq_{pp} \delta' \implies \delta \leq_{ft} \delta' \implies \delta \leq_{do} \delta'.$$
In general, $\delta \leq_{ft} \delta' \not \implies \delta \leq_{pp} \delta'$ and $\delta \leq_{do} \delta' \not \implies \delta \leq_{ft} \delta'$.
\end{prop}

\begin{proof}
In Propositions \ref{prop:pp_implies_ft} and \ref{prop:ft_do}, we conclude that $$\delta \leq_{pp} \delta' \implies \delta \leq_{ft} \delta' \implies \delta \leq_{do} \delta'$$ for any $\delta \vdash g$, $\delta' \vdash g$, and $g \geq 2$. Conversely, Proposition \ref{prop:comparingdeltas} and Corollary \ref{cor:lastdeltaisone} provide criteria and special cases for when $\delta \leq_{ft} \delta'$ implies $\delta \leq_{pp} \delta'$, while Remark \ref{rem:do} shows that $$\delta \leq_{do} \delta' \not \implies \delta \leq_{ft} \delta'$$ in general. Lastly, Example \ref{exmp:Example6} shows that $$\delta \leq_{ft} \delta' \not \implies \delta \leq_{pp} \delta'$$ in general.    
\end{proof}

This concludes the comparison and reveals an interesting property of $\leq_{ft}$, which appears to be a partial order lying between the two classically studied partial orders, $\leq_{pp}$ and $\leq_{do}$, on the set of all partitions of a fixed integer.

\end{document}